\documentclass[12pt,a4paper]{amsart}
\usepackage{amssymb}
\usepackage{tikz}
\usetikzlibrary{shapes}
\usetikzlibrary{calc}
\usetikzlibrary{positioning}
\usetikzlibrary{decorations.markings}

\newcommand{\NN}{{\mathbb{N}}}

\newcommand{\bx} {\mathbf x}
\newcommand{\by} {\mathbf y}
\newcommand{\bz} {\mathbf z}

\newcommand{\cC}{{\mathcal{C}}}

\newcommand{\cH} {\mathcal H}

\newcommand{\cX} {\mathcal X}

\newcommand{\Aut}{{{\operatorname{Aut}}}}
\newcommand{\Alt}{{{\operatorname{Alt}}}}

\newcommand{\ex}{{{\operatorname{ex}}}}

\newcommand{\Gal}{{{\operatorname{Gal}}}}

\newcommand{\orb}{{{\operatorname{orb}}}}

\newcommand{\PGL}{{{\operatorname{PGL}}}}
\newcommand{\PGGL}{{{\operatorname{P}\Gamma \operatorname{L}}}}
\newcommand{\PSL}{{{\operatorname{PSL}}}}

\newcommand{\supp}{{{\operatorname{supp}}}}
\newcommand{\Sym}{{{\operatorname{Sym}}}}

\newcommand{\GL}{\operatorname{GL}}

\newtheorem{thm}{Theorem}[section]
\newtheorem{lem}[thm]{Lemma}
\newtheorem{cor}[thm]{Corollary}
\newtheorem{prop}[thm]{Proposition}

\theoremstyle{definition}
\newtheorem{rem}[thm]{Remark}
\newtheorem{defn}[thm]{Definition}
\newtheorem{exmp}[thm]{Example}

\begin{document}

\title{Tur\'{a}n numbers and switching}


\author{Karen Gunderson}
\address{Department of Mathematics, University of Manitoba}
\email{karen.gunderson@umanitoba.ca}

\author{Jason Semeraro}
\address{Heilbronn Institute for Mathematical Research, Department of
  Mathematics, University of Leicester, United Kingdom}
\email{jpgs1@leicester.ac.uk}

\thanks{The first author gratefully acknowledges funding from NSERC and the second author gratefully acknowledges financial support from the Heilbronn Institute}

\begin{abstract}
Using a switching operation on tournaments we obtain some new lower bounds on the Tur\'{a}n number of the $r$-graph on $r+1$ vertices with $3$ edges. For $r=4$, extremal examples were constructed using Paley tournaments in previous work. We show that these examples are unique (in a particular sense) using Fourier analysis. 

A $3$-tournament is a `higher order' version of a tournament given by an alternating function on triples of distinct vertices in a vertex set.  We show that $3$-tournaments also enjoy a switching operation and use this to give a formula for the size of a switching class in terms of level permutations, generalising a result of Babai--Cameron.
\end{abstract}

\keywords{Tur\'{a}n numbers, tournaments, switching}

\subjclass[2010]{05C65; secondary: 05C35}

\date{\today}

\maketitle

\section{Introduction}\label{s:intro}
Building on the results of \cite{GS17}, we wish to promote the use of tournament switching as a means of constructing hypergraphs with extremal properties. Let $r\geq 2$, $\cH$ be an $r$-uniform hypergraph, or $r$-graph, and $n$ be an integer. The \emph{Tur\'{a}n number} for $\mathcal{H}$, denoted $\operatorname{ex}(\mathcal{H}, n)$, is the maximum number of hyperedges in any $r$-uniform hypergraph on $n$ vertices containing no copy of $\mathcal{H}$. The \emph{Tur\'{a}n density} of $\mathcal{H}$ is $\pi(\mathcal{H}) = \lim_{n \to \infty} \operatorname{ex}(\mathcal{H}, n) \binom{n}{r}^{-1}$. A  well-known averaging argument shows that this limit always exists. 

For $r \ge 2$ we let $\cH(r)$ denote the (unique) $r$-graph on $r+1$ vertices with three edges. Since $\cH(2)$ is just a triangle, we refer to $\cH(r)$ as an \textit{$r$-triangle}. By considering certain arrangements of points on a sphere, it can be shown that $\displaystyle\frac{1}{2^{r-1}} \le \pi(\cH(r))$  (see  \cite[Example 3]{FF84}). It is conjectured that $\pi(\cH(3))=\frac{2}{7}$ (see \cite[Theorem 3]{FF84}) and proven that $\pi(\cH(4))=\frac{1}{4}$ (see Theorem \ref{t:r=4}), but to the best of our knowledge $\frac{1}{2^{r-1}}$ is the best known lower bound for $\cH(r)$ when $r \ge 5$. By developing the theory of tournament switching, we improve upon this lower bound in three cases. Recall that a $t-(n,k,\lambda)$ design is a $k$-graph on $n$ vertices in which every any set of $t$ vertices is contained in exactly $\lambda$ edges.

\begin{thm}\label{t:main}
We have,
$$\frac{9}{64} \le \pi(\cH(6)), \hspace{2mm} \frac{35}{2^{11}} \le \pi(\cH(7)), \hspace{2mm} \mbox{ and } \hspace{2mm} \frac{315}{2^{14}} \le \pi(\cH(8)).$$ Moreover there exists a $5-(12,6,2)$ design with the property that any set of $7$ vertices spans $0$ or $2$ edges and consequently $\ex(\cH(6),12)=264$.
\end{thm}

Two tournaments on the same vertex set  $V$ are \textit{switching equivalent} if one can be obtained from the other by reversing the orientations of all edges incident to some subset of $V$. Any pair of switching equivalent tournaments determine the same \textit{oriented two-graph} (see Lemma \ref{l:onetoone}). 

To prove Theorem \ref{t:main}, we count the number of subtournaments in a random tournament which are switching equivalent to a given tournament $T$ on $r$ vertices (see Theorem \ref{t:rand}). This generalises the approach taken in \cite{GS17} to prove that $\pi(\cH(4))=\frac{1}{4}$ using an augmented Paley tournament (see \cite[Figure 1]{GS17}). For small $r$, and provided the oriented two-graph associated to $T$ is \textit{special} (see Definition \ref{d:special}), we can use $T$ to bound $\pi(\cH(r))$ in terms of the automorphism group of its switching class (see Proposition \ref{p:lb}). Using methods from Fourier analysis, in Section \ref{s:uniq}, we also prove that these examples are unique among all those that arise from \textit{admissible functions}, of which the square character on $\mathbb{F}_q^\times$ supplies an example. 

The paper concludes  by initiating the investigation of a switching operation for \textit{$3$-tournaments}, introduced by Leader and Tan \cite{LT10}. After introducing the operation we use it to count the number of isomorphism classes of elements in a switching class in terms of its automorphism group. 
One hope is that, by analogy with the case $d=2$, $d$-tournaments might provide a rich supply of extremal examples for Tur\'{a}n-type problems.

The problem of determining Tur\'{a}n numbers for hypergraphs related to the $r$-triangles has connections to a number of other extremal problems for hypergraphs.  In their two papers, Brown, Erd\H{o}s, and S\'{o}s~\cite{BES73a, BES73b} considered the parameter $f^{(r)}(n;k, s)$ (elsewhere called $m(n, r, k,s)$), the smallest $m$ so that every $r$-uniform hypegraph with $n$ vertices and $m$ edges has some set of $k$ vertices with at least $s$ edges.  Since there is a unique $r$-uniform hypergraph on $r+1$ vertices with any given number of edges, $\ex(n, \mathcal{H}(r)) = f^{(r)}(n; r+1, 3) - 1$.  Bollob\'{a}s, Leader, and Malvenuto~\cite{BLM11} studied Tur\'{a}n problems for $(s, t)$-daisies, an $r$-uniform hypergraph on $r-t+s$ vertices consisting of all hyperedges containing a fixed $(r-t)$-set.  The $r$-triangle $\mathcal{H}(r)$ is precisely a $\mathcal{D}_r(3, 2)$ daisy.  Reiher, R\"{o}dl and Schacht~\cite{RRS17} determine the Tur\'{a}n density with an additional `uniform density' condition for the hypergraphs $\mathcal{H}(r)$.  Balogh, Clemen, and Lidick\'{y}~\cite{BCL21a} introduced the notion of the extremal co-degree in the $\ell_2$-norm and showed~\cite{BCL21b} that the extremal hypergraphs constructed in~\cite{GS17} are also extremal for the co-degree in the square norm for $\mathcal{H}(4)$-free graphs.

\subsection{Acknowledgements}
This work has taken several years to write up and was mostly completed in 2017 at the University of Bristol while the second author was still a Heilbronn Fellow. Over the years that followed, a number of people have provided input on the project. 
Thanks are also extended to Andy Booker, Tom Oliver, Rob Kurinczuk, Dan Fretwell and Adam Thomas for alerting the second author to the number-theoretic results in Section \ref{s:uniq} on which the proof of Theorem \ref{t:uniq} is based.

\section{Tournaments, two-graphs, switching and automorphisms}\label{s:methods}
\subsection{Tournaments and switching}
Let $T$ be a tournament on a vertex set $V$ regarded as an antisymmetric function $$e_T: \{(x,y) \in V \times V \mid x \neq y \} \rightarrow \{ \pm 1\}, \hspace{2mm} e_T(x,y)=\begin{cases} +1 & \mbox{ if $(x,y) \in E(T)$ } \\ -1 & \mbox{otherwise.}\end{cases}$$ Now define a new function $g_T: \{(x,y,z) \in V \times V \times V \mid \mbox{$x,y,z$ distinct}\} \rightarrow \{\pm 1\}$ via $$g_T(x,y,z):= e_T(x,y)e_T(y,z)e_T(z,x).$$

\begin{defn}[\cite{BC2000},\cite{C77}]\label{d:orienttwo}
Let $V$ be a vertex set. An \textit{oriented two-graph} on $V$ is an alternating function (in the sense that interchanging two arguments changes
the sign) $g: \{(x,y,z) \in V \times V \times V \mid \mbox{$x,y,z$ distinct}\} \rightarrow \{\pm 1\}$ for which \begin{equation}\label{e:eq1}
g(x,y,z)g(y,x,w)g(z,y,w)g(x,z,w)= +1
\end{equation} 
for any four distinct vertices $x,y,x,w$.
\end{defn}

\begin{lem}
For any tournament $T$, $g_T$ is an oriented two-graph.
\end{lem}


The concept of \textit{switching} is used to explain how different tournaments with the same oriented two graph are related:

\begin{defn}[\cite{M95}]\label{def:switch}
Let $V$ be a vertex set.
\begin{enumerate}
\item If $T$ is a tournament on on $V$ and $X \subseteq V$ then \emph{$T$ switched with respect to $X$}, denoted $T^X$ is the tournament obtained from $T$ by reversing the orientations of all edges between $X$ and $V \setminus X$.
\item Two tournaments $T_1$ and $T_2$, both on vertex set $V$, are \emph{switching equivalent} if there exists $X \subseteq V$ so that $T_2$ is precisely $T_1$ switched with respect to $X$. 
\end{enumerate}
\end{defn}
It is straightforward to verify the following:
\begin{lem}
Switching equivalence is an equivalence relation on the set of tournaments on $V$.
\end{lem}

We refer to an equivalence class of tournaments for the switching operation as a \textit{switching class} of tournaments (on $V$).

\begin{defn}\label{d:aug}
Let $T$ be a tournament. The \textit{augmentation of $T$} is the tournament $T^+$ with $V(T^+)= V(T) \cup\{\infty\}$, whose restriction to $V(T)$ is $T$ and with $(x,\infty) \in E(T^+)$ for all $x \in V(T)$. If a tournament is of the form $T^+$ for some $T$, it is said to be \textit{augmented with respect to $\infty$}.
\end{defn}

Note that the vertex labelled $\infty$ in Definition~\ref{d:aug} is an arbitrary new vertex.  The reason for thinking of this as a `vertex at infinity' is connected to the construction of the Paley oriented two-graphs in Section \ref{sec:paley} using projective lines.

\begin{lem}\label{l:onetoone}
Let $V$ be a finite vertex set. There is a one to one correspondence between switching classes of tournaments on $V$ and oriented two-graphs on $V$.
\end{lem}

\begin{proof}
If $T$ is a tournament on $V$ then switching with respect to $X \subset V$ reverses the orientations of evenly many of the edges between a triple of vertices $x,y,z$. Consequently $$g_T(x,y,z)=e_T(x,y)e_T(y,z)e_T(z,x)=e_{T^X}(x,y)e_{T^X}(y,z)e_{T^X}(z,x)=g_{T^X}(x,y,z).$$ On the other hand, if $g$ is an oriented two-graph, the augmented tournament given by fixing an arbitrary  vertex to label $\infty$ and setting $$e_T(x, \infty):=1 \mbox{ and } e_T(x,y):=g(\infty,y,x), \mbox{ for $x,y \neq \infty$ }$$ clearly satisfies $g_T=g$. 
\end{proof}

\begin{rem}\label{rem:ct-equiv-classes} Since switching with respect to $X \subseteq V$ and its complement are equivalent, there are always $2^{|V|-1}$ tournaments in a switching class associated to an oriented two-graph. 
\end{rem}

The following definition will be useful:

\begin{defn} If $g$ is an oriented two-graph on $V$ and $W \subseteq V$,  write $g|_W$ for the oriented two-graph obtained by restricting $g$ to $W$ and call $g|_W$ the \textit{restriction of $g$ to $W$}. 
\end{defn}

The argument in Lemma \ref{l:onetoone} shows that every tournament is switching equivalent to an augmented tournament. In fact, more is true:

\begin{lem}\label{l:univ}
If $T$ is a tournament with vertex set $V$ then for all $0 \leq i \leq |V|-1$ and $v \in V$ there exists $X \subseteq V$ such that $v$ has in-degree $i$ in $T^X$.
\end{lem}

As a consequence we obtain:

\begin{cor}
Suppose that $|V| > 2$. Then there exist at least two isomorphism classes of tournaments in a switching class on $V$.
\end{cor}

\begin{proof}
If $T$ is a tournament with the property that $T^X \cong T$ for each $X \subseteq V$ then by Lemma \ref{l:univ}, the in-degree sequence of $T$ is $\{0,1,2,\ldots,|V|-1\}$. This implies that $T$ is a transitive tournament on $V$. Denote by $v_i$ the vertex of in-degree $i$. If $X:= \{v_0,v_2\}$ then the in-degrees of $v_0$ and $v_2$ in $T^X$ are both $|V|-2$, a contradiction.
\end{proof}

The precise number of tournaments in a switching class is given in Remark \ref{rem:numberinsw} below.

\subsection{Automorphisms of tournaments and switching classes}
The goal of this subsection is Theorem \ref{t:rand} which counts the expected number of subtournaments of a given tournament which lie in a given switching class.

\begin{defn}
If $g$ is an oriented two-graph on $V$ with switching class  $C$ define $$\Aut(g)=\Aut(C)=\{\sigma \in \Sym(V) \mid g(\sigma(x),\sigma(y), \sigma(z))=g(x,y,z), \mbox{ distinct } x,y,z \in V\}.$$ 
\end{defn}

Plainly $\Aut(T_1) \le \Aut(g_{T_1})=\Aut(C)$ for any tournament $T_1$ in a switching class $C$. In fact we have the following result:

\begin{lem}\label{l:cosets}
Let $C$ be a switching class of tournaments on $V$ and fix a tournament $T_1 \in C$. There is a one to one correspondence:

$$\begin{array}{rcl} \displaystyle\Phi:\Aut(C)/\Aut(T_1) &\rightarrow &\{T \in C \mid T \cong T_1\}\vspace{2mm} \\ 
\sigma\Aut(T_1) &\mapsto & \sigma(T_1).\end{array}$$
\end{lem}

\begin{proof}
Let $g$ be the oriented two-graph associated to $C$. Now $\sigma\Aut(T_1)=\sigma'\Aut(T_2)$ if and only if there exists $\tau \in \Aut(T_1)$ such that $\sigma=\sigma'\tau$.  Since $\tau(T_1)=T_1$, we have $$\sigma(T_1)=\sigma'(\tau(T_1))=\sigma'(T_1)$$ so $\Phi$ is injective. Conversely, suppose that $T_1 = \sigma(T_2)$ for some $T_1,T_2 \in C$. Then there exists $X \subset V$ such that $T_2=T_1^X$ and for all distinct $x,y \in V$, we have $$e_{T_1}(x,y)=e_{T_2}(\sigma(x),\sigma(y))=e_{T_1^X}(\sigma(x),\sigma(y)).$$ Therefore for all pairwise distinct elements $x,y,z \in V$, $$\begin{array}{rcl} g_{T_1}(x,y,z) &=& e_{T_1}(x,y)e_{T_1}(y,z)e_{T_1}(z,x) \\&=&e_{T_1^X}(\sigma(x),\sigma(y))e_{T_1^X}(\sigma(y),\sigma(z))e_{T_1^X}(\sigma(z),\sigma(x))\\ &=&g_{T_1^X}(\sigma(x),\sigma(y),\sigma(z))\\ &=& g_{T_1}(\sigma(x),\sigma(y),\sigma(z)) \end{array}$$ and $\sigma \in \Aut(C)$.
\end{proof}

Let us note that two isomorphic tournaments need not be switching equivalent.  As an example, consider the transitive tournament on $3$ vertices: $T_1$ with edges directed $0\to 1$, $0 \to 2$ and $1 \to 2$.  Looking at the four different cases: $T_1$ is switching equivalent to the two other transitive tournaments with edges directed $1 \to 2 \to 0$, or $2 \to 0 \to 1$ and to the cyclic tournament with edges directed $0 \to 2 \to 1 \to 0$.  The other 4 tournaments on $3$ vertices form another switching equivalence class that also contains two transitive tournaments and one cyclic tournament.  

This example extends more generally.  For any $n$, a transitive tournament has no non-trivial automorphisms and so there will be $n!$ different isomorphic tournaments on $n$ vertices.  However, by Remark~\ref{rem:ct-equiv-classes}, there are $2^{n-1}$ different tournaments in a given switching equivalence class.  For $n \geq 3$, $n! > 2^{n-1}$ and so there will always be distinct transitive tournaments that are isomorphic, but not switching equivalent. 

Among other things, the next result shows that every odd divisor of $|\Aut(C)|$ is a divisor of $|\Aut(T)|$ for some $T \in C$:

\begin{lem}\label{l:cosets2}
Let $C$ be a switching class of tournaments on $V$ and let $\{T_1,T_2,\ldots, T_k\}$ be a complete set of isomorphism class representatives of elements of $C$. Then, $$\sum_{i=1}^k \frac{1}{|\Aut(T_i)|}=\frac{2^{|V|-1}}{|\Aut(C)|}.$$
\end{lem}

\begin{proof}
By assumption, $C=C_1 \cup C_2 \cup \cdots C_k$ where $C_i$ is the set of all tournaments in $C$ isomorphic with $T_i$. By Lemma \ref{l:cosets},  $|C_i|=|\Aut(C):\Aut(T_i)|$ so that $$2^{|V|-1}=|C|=\sum_{i=1}^k |C_i|=\sum_{i=1}^k |\Aut(C):\Aut(T_i)| = |\Aut(C)| \cdot \sum_{i=1}^k \frac{1}{|\Aut(T_i)|}. $$ 
\end{proof}

\begin{rem}\label{rem:numberinsw}
For $k$ as in Lemma \ref{l:cosets2}, we actually have $$k=\frac{1}{|\Aut(C)|} \sum_{\substack{\sigma \in \Aut(C) \\ 2 \nmid o(\sigma) }} 2^{c(\sigma)-1},$$ where $c(\sigma)$ denotes the number of cycles in the cycle decomposition for $\sigma$ and $o(\sigma)$ is the order of $\sigma$ (see \cite[Theorem 3.2]{BC2000}). In particular, $|C|=2^{|V|-1}$ if $\Aut(C)=1$, consistent with Lemma \ref{l:cosets2}.
\end{rem}

We now have the following:

\begin{thm}\label{t:rand}
Let $C$ be a switching class of tournaments on $V$. The expected number of subtournaments of a random tournament on $n$ vertices ($n > |V|$) which lie in $C$ is: $${n \choose |V|} \cdot\frac{|\Sym(V):\Aut(C)|}{2^{{|V|-1 \choose 2}}}.$$
\end{thm}

\begin{proof}
Let $\{T_1,T_2,\ldots, T_k\}$ be a complete set of isomorphism class representatives of elements of $C$ and set $r:=|V|$. By Lemma \ref{l:cosets2}, the expected number of elements of $C$ which appear as subtournaments of a random tournament on $n$ vertices is

$${n \choose r} \cdot r! \cdot \left(\frac{1}{2}\right)^{r \choose 2} \cdot \sum_{i=1}^k \frac{1}{|\Aut(T_i)|}={n \choose r} \cdot \frac{|\Sym(r)|}{|\Aut(C)|} \cdot \frac{2^{r-1}}{2^{{r \choose 2}}}={n \choose r} \cdot\frac{|\Sym(r):\Aut(C)|}{2^{{r-1 \choose 2}}},$$ as needed.
\end{proof}

Using the related notion of an \textit{$S$-digraph}, Babai and Cameron characterise $\Aut(C)$ as follows:

\begin{thm}[Babai--Cameron]\label{t:babai}
Let $G$ be a finite group. Then $G=\Aut(C)$ for some switching class of tournaments $C$ if and only if $G$ has cyclic or dihedral Sylow $2$-subgroups.
\end{thm}

\begin{proof}
This follows by combining \cite[Proposition 2.1 and  Theorem 4.1]{BC2000}.
\end{proof}
The following group-theoretic result is therefore relevant. Let $\PGL_2(\mathbb{F}_q)$ denote the quotient of $\GL_2(\mathbb{F}_q)$ by its centre (scalar matrices).   Let SL$_2(\mathbb{F}_q) \le \GL_2(\mathbb{F}_q)$ be the subgroup of matrices with determinant $1$, and $\PSL_2(\mathbb{F}_q)$ denote its image in  $\PGL_2(\mathbb{F}_q)$. Let $\PGGL_2(\mathbb{F}_q)$ be the extension  of $\PGL_2(\mathbb{F}_q)$ by field automorphisms of $\mathbb{F}_q$.

\begin{thm}[Gorenstein--Walter \cite{GW64}]\label{t:gorwalt}
If $G$ is a finite group with dihedral Sylow $2$-subgroups then $G/O(G)$ is either a $2$-group, $\Alt(7)$ or a subgroup of $\PGGL_2(\mathbb{F}_q)$ containing $\PSL_2(\mathbb{F}_q)$, where $O(G)$ denotes the largest normal subgroup of $G$ of odd order.
\end{thm}

Theorems \ref{t:babai} and \ref{t:gorwalt} lead naturally to the question: Do oriented two-graphs with an automorphism group containing $\PSL_2(q)$ actually exist? In the next section we will see that this is indeed the case.
 
\section{The Paley oriented two-graph}\label{sec:paley}

 We will be interested in tournaments constructed from certain types of functions on abelian groups:

\begin{defn}\label{d:ad}
Let $A$ be an abelian group of odd order. A function $f: A \rightarrow \{0,\pm 1\}$ is \textit{admissible} if $f(a)=0$ if and only if $a=0$ and  $f(-a)=-f(a)$  for all $a \in A$.
\end{defn} 
Clearly $\displaystyle \sum_{a \in A} f(a)=0$ for an admissible function $f$.
\begin{defn}
Suppose $f$ is an admissible function on $A$.  Let $T_{f}$ be the tournament with  $V(T_f)=A$ and with $(x,y) \in E(T)$ if and only if $f(y-x)=1$. 
\end{defn}







Let $q$ be a prime power and define a function $$\begin{array}{rcl} \chi: (\mathbb{F}_q,+) & \rightarrow & \{0,\pm 1\} \\ x &\mapsto &\begin{cases} 
0 & \mbox{ if $x=0$ } \\
1 & \mbox{ if there exists $y \in \mathbb{F}_q$ with $y^2 = x$ in $\mathbb{F}_q^\times$}\\ -1 & \mbox{ otherwise.}\end{cases}\end{array}$$ Note, in particular that $\chi(x)=-\chi(-x)$ for $x \in \mathbb{F}_q^\times$ so $\chi$ is an admissible function on $\mathbb{F}_q$.

\begin{defn}\label{d:paley}
Let $q \equiv 3 \pmod 4$ be a prime power and let $\chi$ be as above. 
\begin{enumerate}
\item The \textit{Paley Tournament} is the tournament $T_\chi$ on $\mathbb{F}_q$.
\item The \textit{Paley oriented two-graph} $g_q$ is defined by $g_q:=g_{T_\chi^+}$. 
\end{enumerate}
\end{defn}

Recall that the  \textit{projective line}  $\mathbb{P}^1 \mathbb{F}_q$ is the the set of equivalence classes $\{\bx=(x_1,x_2) \mid x_i \in \mathbb{F}_q\}/\sim$ with $\bx \sim \by$ if there exists $\lambda \in \mathbb{F}_q^\times$ such that $\lambda \bx = \by$. We have the following result: 

\begin{lem}\label{l:chiequalsq} $g_q$ is isomorphic to the oriented two-graph $g$ on $\mathbb{P}^1 \mathbb{F}_q$ given by $$g(\bx,\by,\bz)= \chi(\det(\bx:\by)\det(\by:\bz)\det(\bz:\bx)).$$  
\end{lem} 

\begin{proof}

Writing $V(T_\chi^+) = \mathbb{F}_q \cup \{\infty\}$ and $\mathbb{P}^1 \mathbb{F}_q =\{(x,1) \mid x \in \mathbb{F}_q \} \cup \{(1,0)\}$, we have a natural identification: 
$$V(T_\chi^+) \rightarrow \mathbb{P}^1 \mathbb{F}_q, \hspace{5mm} x \mapsto (x,1), \hspace{1mm} \infty \mapsto (1,0)$$ and equalities

 $\chi \left(\det\left(\begin{matrix} x & 1 \\ y & 1  \end{matrix} \right)\right) =\chi(x-y) \mbox{  and } \chi \left(\det\left(\begin{matrix} 1 & 0 \\ x & 1  \end{matrix} \right)\right) = 1 \mbox{ for distinct $x,y \in \mathbb{F}_q$. }$

 It remains to show that $g$ is independent of the choice of representatives $\bx,\by,\bz.$ If $\bx \sim \bx'$, $\by \sim \by'$ and $\bz \sim \bz'$ then there exist $\lambda,\mu, \nu \in \mathbb{F}_q^\times$ such that $\bx=\lambda \bx'$, $\by=\mu \by'$ and $\bz=\nu \bz'$. Now $$\begin{array}{rcl} g_q(\bx,\by,\bz) &=& \chi(\det(\bx:\by)\det(\by:\bz)\det(\bz:\bx)) \\ &=& \chi(\det(\lambda\bx':\mu\by')\det(\mu\by':\nu\bz')\det(\nu\bz':\lambda\bx)) \\ &=& \chi(\lambda^2\mu^2\nu^2\det(\bx':\by')\det(\by':\bz')\det(\bz':\bx')) \\ &=& \chi(\det(\bx':\by')\det(\by':\bz')\det(\bz':\bx')) \\ &=& g_q(\bx',\by',\bz').\end{array} $$
\end{proof}

This interpretation of $g_q$ facilitates the computation of its automorphism group. Recall that $\PGL_2(\mathbb{F}_q)$ acts naturally on $\mathbb{P}^1\mathbb{F}_q$ by right multiplication.

\begin{prop}\label{p:auto}
We have, $\PSL_2(\mathbb{F}_q) \le \Aut(g_q)$.
\end{prop}

\begin{proof}
Let $\bx:=(x_1,x_2), \by:=(y_1,y_2) \in \mathbb{P}^1 \mathbb{F}_q$ and $\sigma:=  \left(\begin{smallmatrix}
a & b  \\
c & d   \end{smallmatrix}\right) \in \GL_2(q)$. Denote by $\bar{\sigma}$ the image of $\sigma$ in $\PGL_2(q)$. We have,

$$\begin{array}{rcl}
\det(\bx\bar{\sigma},\by\bar{\sigma}) & = & \det((ax_1+cx_2,bx_1+dx_2):(ay_1+cy_2,by_1+dy_2)) \\
 & =& (ax_1+cx_2)(by_1+dy_2)-(bx_1+dx_2)(ay_1+cy_2) \\
 & = & (ad-bc)(x_1y_2-x_2y_1) \\
 & = & \det(\sigma)\det(\bx:\by). 
\end{array}$$

Hence if $\bx,\by,\bz \in \mathbb{P}^1(\mathbb{F}_q)$ are distinct elements, we have 
$$\begin{array}{rcl}
\chi(\det(\sigma)) g_q(\bx,\by,\bz) & = & \chi(\det(\sigma))\chi(\det(\bx:\by)\det(\by:\bz)\det(\bz:\bx)) \\
 & = & \chi(\det(\bx\bar{\sigma}:\by\bar{\sigma})\det(\by\bar{\sigma}:\bz\bar{\sigma})\det(\bz\bar{\sigma}:\bx\bar{\sigma})) \\
 & = & g_q(\bx\bar{\sigma},\by\bar{\sigma},\bz\bar{\sigma}). 

\end{array}$$ 

In particular, $\sigma \in \Aut(g_q)$ if $\det(\sigma)=1$ and therefore $\PSL_2(q) \le \Aut(g_q)$.
\end{proof}

\begin{rem}
 It is a consequence of Proposition \ref{p:auto} that for any oriented two-graph $h$ on $m < q+1$ vertices,  the set $\{W \subset \mathbb{P}^1 \mathbb{F}_q \mid g_q|_W \cong h\}$ is a union of orbits of $\PSL_2(q)$ and hence known by \cite{CMOT-R06}. This observation will be used later.
\end{rem}

\section{Tur\'{a}n results for $r$-triangles}
Recall the following from Section \ref{s:intro}:
\begin{defn}
For $r \ge 2$, an \textit{$r$-triangle} $\mathcal{H}(r)$ is any $r$-graph on $r+1$ vertices with $3$ edges. 
\end{defn}

We have the following generalisation of Mantel's theorem:

\begin{prop}{\cite[Proposition 14]{GS17}}\label{p:decaen}
Let $r \ge 2$ and $\cH$ be an $r$-graph and assume that $\cH$ is $\cH(r)$-free. Then $$|E(\cH)| \le \frac{n}{r^2}{n \choose r-1}$$ with equality if and only if every set of $r-1$ vertices is contained in exactly $n/r$ hyperedges.
\end{prop}

In \cite{FF84} Frankl and F\"{u}redi characterise $3$-graphs for which every set of $4$ vertices spans either $0$ or $2$ hyperedges (such $3$-graphs are in particular $\cH(3)$-free). Using this they prove that $\pi(\cH(3)) \ge \frac{2}{7}$. Another consequence of their results is the following:

\begin{prop}\label{p:decaenbounds}
We have, 
\begin{equation}\label{e:dec} \frac{1}{2^{r-1}} \le \pi(\cH(r)) \le \frac{1}{r},\end{equation}
\end{prop} 
\begin{proof}
The upper bound is immediate from Proposition \ref{p:decaen}, while the lower bound follows from \cite[Example 3]{FF84}.
\end{proof}

The following Proposition, which follows immediately from Theorem \ref{t:rand}, has the potential to improve on the lower bound in Proposition \ref{p:decaenbounds} when $r \le 8$:

\begin{prop}\label{p:lb}
Let $g$ be an oriented two-graph on $r$ vertices with the property that for any oriented two-graph $h$ with vertex set $V$ of size $r+1$, $$|\{A \subset V \mid h|_A \cong g\}| \le 2.$$ Then $$\frac{r!}{|\Aut(g)|2^{r-1 \choose 2}} \le \pi(\cH(r)).$$ 
\end{prop}

 \begin{defn}\label{d:special}
An oriented two-graph $g$ on $r$ vertices which satisfies the hypothesis of Proposition \ref{p:lb}  is called \textit{special}. \end{defn} 

In the next two subsections, we determine all special oriented two-graphs for $4 \le r \le 8$ which lead to an improvement in the lower bound provided by Proposition \ref{p:decaenbounds}.

\subsection{The case $r=4$} There are two oriented two-graphs on $4$ vertices, one of which (the Paley oriented two-graph $g_3$) is special:

\begin{lem}\label{l:spec}
$g_3$ is special and  $\Aut(g_3) \cong \PSL_2(\mathbb{F}_3)$. \end{lem}

\begin{proof}
Let $h$ be an oriented two-graph with vertex set $V$ of size $5$ and consider the set $E=\{X \subset V \mid h|_X \cong g_3\}$. If $|E| > 0$ then $h|_X \cong g_3$ for some $X \subset V$ so $h \cong g_{T_3^+}$ where $T_3^+$ denotes the augmentation of $T_3$ and $|E|=2$. By Proposition \ref{p:auto}, $\Aut(g_3) \ge \PSL_2(\mathbb{F}_3)$. There are two isomorphism classes of tournaments in the switching class of $g_3$, both with an automorphism group of order $3$. Hence by Lemma \ref{l:cosets2}, $$|\Aut(g_3)| \cdot \left( \frac{1}{3}+\frac{1}{3} \right) = 2^3$$ so $|\Aut(g_3)|=12$ and the lemma follows.
\end{proof}


Define $\cH_q:=\{X \subset \mathbb{P}^1 \mathbb{F}_q \mid g_q|_X \cong g_3\}$. We can rephrase the main result of \cite{GS17} as follows:

\begin{thm}\label{t:r=4}
We have $\pi(\cH(4))=\frac{1}{4}$ and for any odd prime power $q \equiv 3 \pmod 4$, $\cH_q$ is $\cH(4)$-free and attains the upper bound on edges in Proposition~\ref{p:decaen}. Hence $\ex(\cH(4),q+1)=\frac{q+1}{16}{q+1 \choose 3}$. 
\end{thm}

\begin{proof}
Since $g_3$ is special by Lemma \ref{l:spec}, $\frac{1}{4} \le \pi(\cH(4))$ by Proposition \ref{p:lb}. Hence   $\pi(\cH(4))=\frac{1}{4}$ by Proposition \ref{p:decaenbounds}. The exact result is a consequence of  the fact that $\cH_q$ is $\cH(4)$-free by Lemma \ref{l:spec} and extremal with respect to this property by a counting argument (see \cite[Theorem 11]{GS17} or the proof of Lemma \ref{l:admissext}). 
\end{proof}

In Theorem \ref{t:uniq} below, we will see that, when $q \equiv 3 \pmod 4$ is prime,  $\cH_q$ is unique among $4$-graphs constructed from admissible functions which produce hypergraphs that are $\cH(4)$-free and attain the upper bound on edges in Proposition~\ref{p:decaen}. 

\begin{rem}
Among other results, in \cite{BBLZ20} the authors show that one can construct a $4$-graph which is $\cH(4)$-free and attains the  bound in Proposition~\ref{p:decaen} from any skew Hadamard matrix (or skew-conference matrix) of order $n \equiv 0 \pmod 4$ (such matrices are conjectured to exist for all such $n$). This extends Theorem \ref{t:r=4} and yields an abundance of extremal examples which do \textit{not} necessarily arise from tournaments constructed from admissible functions.

There is a correspondence between skew-Hadamard matrices and a class tournaments that include the Paley tournaments.  Brown and Reid~\cite{BR72} introduced the notion of a \emph{doubly regular tournament}, which is a tournament in which every vertex has the same out-degree and every pair of vertices have the same number of common out-neighbours.  They showed that if such a tournament of order $n$ exists, then the out-degree of every vertex is $\frac{n-1}{2}$ and the number of common out-neighbours of any pair is $\frac{n-3}{4}$.  Necessarily, $n \equiv 3 \pmod{4}$.  The Paley tournaments are an example of doubly-regular tournaments.  Brown and Reid showed that there is a doubly-regular tournament of order $n$ if and only if there is a skew-Hadamard matrix of order $n+1$.  

Thus, extending the results of~\cite{GS17}, the results of Belkouche, Boussa\"{i}ri, Lakhlifi and Zaidi~\cite{BBLZ20} show that for any $n$ for which there exists a skew-Hadamard matrix of order $n$, $\ex(\cH(4), n) = \frac{n}{16}\binom{n}{3}$ and that the lower bound is attained by Baber's construction applied to a doubly-regular tournament on $n-1$ vertices augmented with a vertex with all incident edges directed towards it.

Interestingly, in~\cite{BBLZ20}, it is shown that for $n \equiv 3 \pmod{4}$, the hypergraph constructed in a similar manner from a tournament of order $n$ will have the maximum number of edges exactly when the tournament is switching equivalent to a doubly-regular tournament.  They give closely related results for $n \equiv 2 \pmod{4}$, characterising tournaments that achieve an upper bound in terms of the design-type properties of the Seidel (or signed) adjacency matrix.
\end{rem}

\subsection{The cases $5 \le r \le 8$}

For the cases $r = 5, 6, 7, 8$, an exhaustive search with SAGE was used to find examples of tournaments giving lower bounds on $\pi(\cH(r))$ using Proposition~\ref{p:lb}.  SAGE has a built-in iterator over all non-isomorphic tournaments on a given number of vertices using Nauty and these were used to generate all switching classes for tournaments on a given number of vertices and the number of automorphisms of each representative in a switching class.  In order to determine whether a given switching and its associated two-graph is special, the following lemma was used to reduce the number of cases to check.  Recall that for a tournament $T$, the tournament $T^+$ is obtained from $T$ by adding a new vertex $\infty$, with all incident edges directed towards the vertex $\infty$.

\begin{lem}\label{l:counting-copies}
Let $g$ be an oriented two-graph on $V$ and let $C$ be its associated switching class of tournaments on $V$.  The two-graph $g$ is special if{f} for every $T \in C$, there is at most one vertex $v \in V$ with $T^+ - v \in C$.
\end{lem}

\begin{proof}
Set $|V| = r$.  By translating oriented two-graphs to tournaments, one can see that $g$ is special if{f} for every tournament $L$ on $r+1$ vertices, there are at most $2$ $r$-sets of vertices that induce a tournament isomorphic to an element of $C$.  Thus, it is clear that this condition is necessary for a two-graph to be special.  

To see that it is sufficient, let $L$ be any tournament on $r+1$ vertices.  If $L$ has no $r$-sets that induce a tournament isomorphic to an element of $C$, there is nothing to prove.  Therefore, assume that there is a vertex $w$ in $L$ so that $L-w$ is isomorphic to a tournament $T_1 \in C$. There is a tournament $L'$ that is switching-equivalent to $L$ with the property that, in $L'$, every edge incident to $w$ is directed towards $w$.  Since $C$ is an equivalence class for the switching operation, the number of $r$-sets of vertices in $L$ that induce a tournament isomorphic to an element of $C$ is exactly the same as the number of $r$-sets of vertices in $L'$ with this property.  Furthermore, $L' -w$ is isomorphic to some tournament $T_1' \in C$ and so $L' \cong (T_1')^+$.  Therefore, there are at most $2$ $r$-sets of vertices in $L'$ that induce a tournament isomorphic to an element of $C$ and the same holds for $L$.
\end{proof}

Using Lemma~\ref{l:counting-copies}, it is straightforward to determine if an oriented two-graph is special from the list of tournaments in the corresponding switching equivalence class.

For $r=5$ there are no special oriented two-graphs.

For $r=6$ there are two special oriented two-graphs with the best lower bound provided by the switching class $g$ of the tournament in Figure \ref{fig:tourn6}.

\begin{figure}[htb]
\caption{A special tournament on $6$ vertices}
\begin{center}
\begin{tikzpicture}
	[decoration={markings, mark=at position 0.6 with {\arrow{>}}}] 
	\tikzstyle{vertex}=[circle, draw=black,  minimum size=5pt,inner sep=0pt]
	
	\node[vertex, label=above:{}] at (2, 0) (y) {};
	\node[vertex, label=above:{}] at (2*0.309, 2*0.95) (z) {};
	\node[vertex, label=left:{}] at (2*-0.809, 2*0.588) (w) {};
    \node[vertex, label=below:{}] at (2*-0.809, 2*-0.588) (a) {};
	\node[vertex, label=right:{}] at (2*0.309, 2*-0.95) (b) {};	
    \node[vertex, label=right:{}] at (0, 0) (c) {};		
	
	\draw[postaction={decorate}] (y) -- (z);
	\draw[postaction={decorate}] (z) -- (w);
	\draw[postaction={decorate}] (w) -- (a);
	\draw[postaction={decorate}] (a) -- (b);
	\draw[postaction={decorate}] (b) -- (y);
	
	\draw[postaction={decorate}] (z) -- (b);
	\draw[postaction={decorate}] (b) -- (w);
	\draw[postaction={decorate}] (w) -- (y);
	\draw[postaction={decorate}] (y) -- (a);
	\draw[postaction={decorate}] (a) -- (z);
	
	\draw[postaction={decorate}] (c) -- (b);
	\draw[postaction={decorate}] (c) -- (w);
	\draw[postaction={decorate}] (c) -- (y);
	\draw[postaction={decorate}] (c) -- (a);
	\draw[postaction={decorate}] (c) -- (z);
	
\end{tikzpicture} \hspace{10pt}
\end{center}

\label{fig:tourn6}
\end{figure}
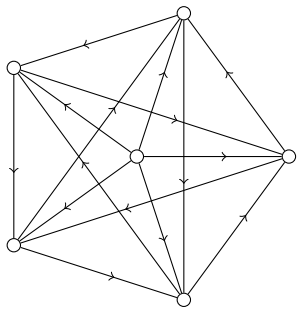

\begin{lem}
$g$ is special and $\Aut(g)$ has order $5$.
\end{lem}
\begin{proof}
Using a computer, we calculate that $g$ is special and that there are 8 isomorphism classes of tournaments in $C$ with automorphism groups of orders $5,5,1,1,1,1,1,1$. Hence $\Aut(C)$ has order $5$ by Lemma  \ref{l:cosets2}. \end{proof}

\begin{cor}
We have $\frac{9}{64} \le \pi(\cH(6)) \le \frac{1}{6}$ and $\ex(\cH(6),12)=264$.
\end{cor}

\begin{proof}
Since $g$ is special, $\frac{9}{64} \le \pi(\cH(6)) \le \frac{1}{6}$ by Proposition \ref{p:lb} and Proposition \ref{p:decaenbounds}. The $6$-graph constructed from $g_{11}$, $\{A \subset V(g_{11}) \mid g_{11}|_A \cong g\}$ is  $\cH(6)$-free and extremal with respect to this property by Proposition \ref{p:decaen}. 
\end{proof}

We speculate that this example is related to the existence of a $4$-tournament on $n$ vertices with at least $\frac{9}{64}{n \choose 5}$ directed $5$-simplices established by Leader and Tan (see \cite[Section 4]{LT10}]).

For $r=7$ there is a unique special oriented two-graph with an automorphism group isomorphic to $C_3 \times C_3$.  One tournament in the switching class of this oriented two-graph is given in Figure~\ref{fig:tourn7}. Using SAGE, it was verified that this oriented two-graph is special and the switching class has 16 tournaments with automorphism groups of order 1, 1, 1, 1, 3, 3, 3, 3, 3, 3, 3, 3, 9, 9, 9, 9.  This gives the following.

\begin{figure}[htb]
\caption{A special tournament on $7$ vertices}
\begin{center}
\begin{tikzpicture}
	[decoration={markings, mark=at position 0.6 with {\arrow{>}}}] 
	\tikzstyle{vertex}=[circle, draw=black,  minimum size=5pt,inner sep=0pt]
		
	\foreach \x/\y in {0/right, 1/above, 2/above, 3/left, 4/below, 5/below}
	{
		\node[vertex, label=\y:{$\x$}] at ({60*\x}:2cm) (\x) {};
	}
	
	\node[vertex, label=above:{$6$}] at (-4, 0) (6) {};
	\foreach \x in {0, 1, 2, 3, 4, 5}
	{
	\draw[->] (6) -- ++(1, {1.5-3*\x/5});
	}
	
	\draw[postaction={decorate}] (0) --(2);
	\draw[postaction={decorate}]  (1) -- (0); 
	\draw[postaction={decorate}] (2) -- (1);
	\draw[postaction={decorate}] (3) -- (0);
	\draw[postaction={decorate}] (3) -- (1);
	\draw[postaction={decorate}] (3) -- (2);
	\draw[postaction={decorate}] (3) -- (5);
	\draw[postaction={decorate}] (4) -- (0);
	\draw[postaction={decorate}] (4) -- (1);
	\draw[postaction={decorate}] (4) -- (2);
	\draw[postaction={decorate}] (4) -- (3);
	\draw[postaction={decorate}] (5) -- (0);
	\draw[postaction={decorate}] (5) -- (1);
	\draw[postaction={decorate}] (5) -- (2);
	\draw[postaction={decorate}] (5) -- (4);

\end{tikzpicture} \hspace{10pt}
\end{center}

\label{fig:tourn7}
\end{figure}
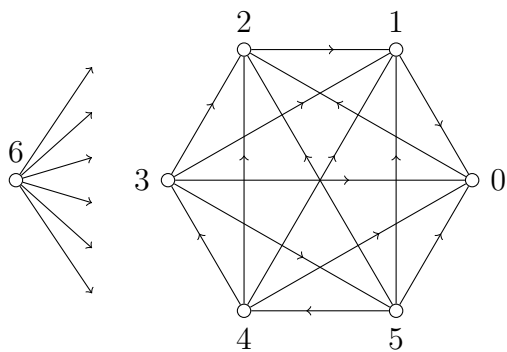

\begin{cor}
We have $\frac{1}{2^6} < \frac{35}{2^{11}} \le \pi(\cH(7))$.
\end{cor}

\begin{proof}
By Lemma~\ref{l:cosets2}
\[
\frac{2^6}{|\Aut(C)|} = 4\cdot \frac{1}{1} + 8 \cdot \frac{1}{3} + 4 \cdot \frac{1}{9} = \frac{64}{9}.
\]
Thus, $|\Aut(C)| = 2^6 \cdot \frac{9}{64} = 9$ and by Proposition~\ref{p:lb},
\[
\pi(\cH(7)) \geq \frac{7!}{9 \cdot 2^{15}} = \frac{35}{2^{11}}.
\]
\end{proof}



For $r = 8$, the best lower bound for $\pi(\cH(8))$ using Proposition~\ref{p:lb} would be provided by a special oriented two-graph with a trivial automorphism group. Using SAGE it was verified that there are $40$ switching classes with trivial automorphism group of which $9$ are special. A tournament representing such a class is provided in Figure \ref{fig:tourn8}.  This gives the following lower bound.

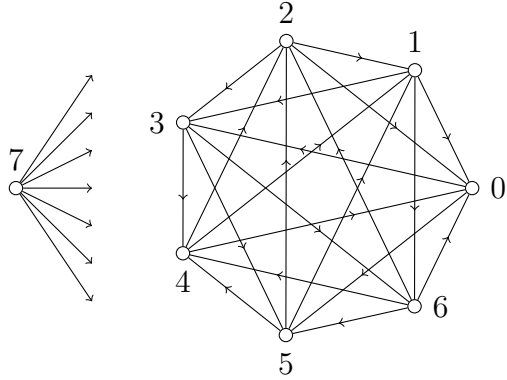
\begin{figure}[htb]
\caption{A special tournament on 8 vertices}
\begin{center}
\begin{tikzpicture}
	[decoration={markings, mark=at position 0.6 with {\arrow{>}}}] 
	\tikzstyle{vertex}=[circle, draw=black,  minimum size=5pt,inner sep=0pt]
		
	\foreach \x/\y in {0/right, 1/above, 2/above, 3/left, 4/below, 5/below, 6/right}
	{
		\node[vertex, label=\y:{$\x$}] at ({51.4*\x}:2cm) (\x) {};
	}
	
	\node[vertex, label=above:{$7$}] at (-4, 0) (7) {};
	\foreach \x in {0, 1, 2, 3, 4, 5, 6}
	{
	\draw[->] (7) -- ++(1, {1.5-\x/2});
	}
	
	\draw[postaction={decorate}] (0)--(3);
	\draw[postaction={decorate}]  (0) -- (5);
	\draw[postaction={decorate}]  (1) -- (0);
	\draw[postaction={decorate}]  (1) -- (3);
	\draw[postaction={decorate}]  (1) -- (6);
	\draw[postaction={decorate}]  (2) -- (0);
	\draw[postaction={decorate}]  (2) -- (1);
	\draw[postaction={decorate}]  (2) -- (3);
	\draw[postaction={decorate}]  (3) -- (4);
	\draw[postaction={decorate}]  (3) -- (5);
	\draw[postaction={decorate}]  (3) -- (6);
	\draw[postaction={decorate}]  (4) -- (0);
	\draw[postaction={decorate}]  (4) -- (1);
	\draw[postaction={decorate}]  (4) -- (2);
	\draw[postaction={decorate}]  (5) -- (1);
	\draw[postaction={decorate}]  (5) -- (2);
	\draw[postaction={decorate}]  (5) -- (4);
	\draw[postaction={decorate}]  (6) -- (0);
	\draw[postaction={decorate}]  (6) -- (2);
	\draw[postaction={decorate}]  (6) -- (4);
	\draw[postaction={decorate}]  (6) -- (5);
\end{tikzpicture} \hspace{10pt}
\end{center}

\label{fig:tourn8}
\end{figure}

\begin{cor}
We have
$\frac{1}{2^7} < \frac{315}{2^{14}} \le \pi(\cH(8)).$
 \end{cor}
 
 \begin{proof}
 Using Lemma~\ref{l:cosets2}, $|\Aut(C)| = 1$ and so by Proposition~\ref{p:lb},
 \[
 \pi(\cH(8)) \geq \frac{8!}{1 \cdot 2^{21}} = \frac{315}{2^{14}}.
 \]
 \end{proof}

\section{A uniqueness result when $r=4$}\label{s:uniq}
Suppose that $f$ is an an admissible function on an abelian group $A$ of odd order and define:
$$\cH_f=\{X \subseteq V(T_f^+) \mid g_{T_f^+}|_X = g_3\}.$$ Note, in particular, that $\cH_\chi=\cH_q$ is the $4$-graph in Theorem \ref{t:r=4} above. In this section, we prove the following result:

\begin{thm}\label{t:uniq}
Let $p$ be prime with $p \equiv 3 \pmod 4$ and suppose that $f$ is an admissible function on $\mathbb{F}_p$. If $H_f$ is $\cH(4)$-free and attains the upper bound on edges in Proposition~\ref{p:decaen},  then $\cH_f=\cH_p$.
\end{thm}

We require some notions from Fourier analysis:

\begin{defn}
Let $A$ be an abelian group, let  $f,g: A \rightarrow \mathbb{C}$ be two functions on $A$ and let $\rho: A \rightarrow \mathbb{C}^\times$ be a representation.
\begin{enumerate}
\item  the \textit{convolution} of $f$ and $g$ at $b \in A$ is $(f \star g)(b) = \displaystyle\sum_{a \in A} f(b-a)g(a).$
\item the \textit{Fourier transform} of $f$ at $\rho$ is $\displaystyle \sum_{a \in A} f(a)\overline{\rho(a)}$
\end{enumerate}
\end{defn}

It is well known that 
$$\widehat{f \star g}(\rho)=\widehat{f}(\rho)\widehat{g}(\rho) \mbox{ and } f(a)=\frac{1}{|A|} \sum_{\rho \in \widehat{A}} \widehat{f}(\rho) \rho(a), \mbox{ for all $a \in A$},$$
where $\widehat{A}$ denotes the set of all linear characters of $A$. In particular $f=0$ if $\widehat{f}(\rho)=0$ for some character $\rho$. 
\begin{lem}\label{l:conv}
Let $f: \mathbb{F}_p \rightarrow \{0,\pm1\}$ be any function with $f(x)=0$ if and only if $x=0$ for which $f \star f = \chi \star \chi$. Then $f = \pm \chi$. 
\end{lem}

 There are at least two ways to prove this:

\begin{proof}[First proof]
This goes via the \textit{uncertainty principle}. For a function $g: \mathbb{F}_p \rightarrow \mathbb{C}$, write $$\supp(g):=\{x \in \mathbb{F}_p \mid g(x) \neq 0\} \mbox{ and } \supp(\widehat{g}):=\{\rho \in \widehat{\mathbb{F}_p} \mid \widehat{g}(\rho) \neq 0\} $$ for the \textit{support} of $g$ and $\hat{g}$ respectively. Then, by \cite[Theorem 1.1]{T03}, \begin{equation}\label{e:tao}
|\supp(g)|+|\supp(\widehat{g})| \ge p+1. 
\end{equation} 

Set $F:=f-\chi$ and $G:=f+\chi$ and suppose that $F$ and $G$ are both non-zero. By (\ref{e:tao}), $$|\supp(F)|+|\supp(\widehat{F})|+|\supp(G)|+|\supp(\widehat{G})| > 2p.$$

Now $FG(x)=0$ since $f^2(x)=\chi^2(x)=1$ for all non-zero $x \in \mathbb{F}_p$ so by the pigeonhole principle, $$|\supp(F)|+|\supp(G)| \le p.$$ Also, $F \star G=0$ by the commutativity of $\star$ and hypothesis so $0=\widehat{F \star G}=\widehat{F}\widehat{G}$ and $$|\supp(\widehat{F})|+|\supp(\widehat{G})| \le p.$$ This delivers the required contradiction.     
\end{proof}

We are grateful to Andy Booker for communicating the following alternative argument which uses Gauss sums.

\begin{proof}[Second proof]
Let $\zeta_p$ be a $p$'th root of unity and recall the isomorphism $$\mathbb{F}_p^\times \rightarrow \Gal(\mathbb{Q}(\zeta_p)/\mathbb{Q}), \mbox{ where } a \mapsto \sigma_a: \zeta_p \mapsto \zeta_p^a.$$ Let $a \in \mathbb{F}_p^\times$ be such that $\chi(a)=1$ and observe that $\sigma_a$ leaves invariant the Gauss sum $$\widehat{\chi}(\rho)=\sum_{x \in \mathbb{F}_p^\times} \chi(x)\zeta_p^x=\sqrt{\chi(-1)p}, \mbox{ where } \rho(x)=\zeta_p^x.$$  

Now $\widehat{f}(\rho)^2=\widehat{f\star f}(\rho) =\widehat{\chi\star \chi}(\rho)= \widehat{\chi}(\rho)^2$ by hypothesis, so $\widehat{f}(\rho)=\pm \widehat{\chi}(\rho)$ and $\sigma_a$ must also leave $\widehat{f}(\rho)$ invariant. In particular, $$\widehat{f}(\rho)=\sum_{x \in \mathbb{F}_p^\times} f(x)\zeta_p^x= \sum_{x \in \mathbb{F}_p^\times} f(x)\zeta_p^{ax}=\sum_{x \in \mathbb{F}_p^\times} f(xa^{-1})\zeta_p^x$$ and we must have $f(x)=f(xa^{-1})$ for all $x \in \mathbb{F}_p^\times$. In particular when $x=1$, $f(1)=f(a^{-1})$ so $f(a)$ is constant $a \in \mathbb{F}_p^\times$ with $\chi(a)=1$. Equivalently,  $f=\pm \chi$, as required.
\end{proof}

\begin{lem}\label{l:admissext}
Suppose that $f$ is an admissible function on an abelian group $A$ with $|A| \equiv 3 \pmod 4$. If $\cH_f$ is $\cH(4)$-free and attains the upper bound on edges in Proposition~\ref{p:decaen} for $\cH(4)$, then $$(f \star f)(x)= \begin{cases} 1 & \mbox{ if $x \neq 0$ } \\ |A|-1 & \mbox{ if $x = 0$. } \end{cases}$$
\end{lem}

\begin{proof}
By the hypothesis on $\cH_f$, any three vertices are contained in exactly $(|A|+1)/4$ hyperedges by Proposition \ref{p:decaen}. In particular, for any vertex $u$ with $f(u)=1$, we must have $$\frac{1}{4}\sum_{a \in A \backslash \{0,u\}} (1-f(a))(1-f(u-a)) = \frac{|A|+1}{4}.$$ Expanding, and using that $f$ is admissible we have $$\begin{array}{rcl} |A|+1&=& \displaystyle\sum_{a \in A \backslash \{0,u\}} (1-f(a)-f(u-a)+f(a)f(u-a)) \\ &=& (|A|-2)+2f(u)+(f \star f)(u), \end{array}$$ from which we deduce that $(f \star f)(u)=1$. Since $f$ is an odd function, we also have $(f \star f)(-u)=1$, and the lemma follows.  
\end{proof}

We now prove Theorem \ref{t:uniq}.

\begin{proof}[Proof of Theorem \ref{t:uniq}]
By the hypothesis on $\cH_f$,   $f \star f = \chi \star \chi$ by Lemma \ref{l:admissext} so  $f = \pm \chi$ by Lemma \ref{l:conv}. Hence $\cH_p \cong \cH_\chi$ (since $\cH_\chi \cong \cH_{-\chi}$  by Lemma \ref{l:chiequalsq}).
\end{proof}


\section{$3$-tournaments and switching}

In \cite{LT10}, Leader and Tan introduce the notion of a $d$-tournament for $d \ge 2$. A $2$-tournament is just a tournament, and for $d=3$ we have the following definition:

\begin{defn}
Let $V$ be a vertex set. A $3$-tournament is any alternating function $g : \{(x,y,z) \in V \times V \times V \mid x,y,z$ distinct $\} \rightarrow \{\pm 1\}$.
\end{defn}

Thus an oriented two-graph  is just a particular example of a $3$-tournament. Recall that a \textit{two-graph} is a $3$-graph $\cX$ on $V$ for which zero, two or four $3$-subsets of any $4$-subset of $V$ lie in $\cX$. 
The number of two-graphs on $V$ is easily seen to be $2^{{|V|-1}\choose{2}}$.

\begin{defn}
Let $V$ be a vertex set.
\begin{enumerate}
\item If $g$ is a $3$-tournament on $V$ and $\cX$ is a two-graph on $V$ then $g$ switched with respect to $\cX$, denoted $g^\cX$, is the $3$-tournament obtained via
$$g^\cX(x,y,z) = \begin{cases} g(x,y,z) & \mbox{ if $\{x,y,z\} \in \cX$ } \\ -g(x,y,z) & \mbox{ otherwise.} \end{cases}$$
\item Two $3$-tournaments $g_1$ and $g_2$ on $V$ are \textit{switching equivalent} if there exists a two-graph $\cX$ on $V$ with $g_1^\cX = g_2$.
\item A \textit{switching class} $D$ of $3$-tournaments on $V$ is a set of $3$-tournaments which is closed under the operation of switching.
\end{enumerate} 
\end{defn}
If $g$ is a $3$-tournament, define $$\cH(g):=\{\{x,y,z,w\} \mid g(x,y,z)g(y,x,w)g(z,y,w)g(x,z,w)=+1\}.$$
To see this is well-defined, observe that (since $g$ is alternating) the condition on $g$ remains true after applying any permutation of $\{x,y,z,w\}$. We note the following:

\begin{lem}
If $g_1$ and $g_2$ are switching equivalent $3$-tournaments, then $\cH(g_1)=\cH(g_2)$.
\end{lem}

Note also that a $3$-tournament $g$ is an oriented two-graph if and only if $\cH(g)$ is complete. 

\begin{defn}
If $g$ is a $3$-tournament with switching class $\cC$ define $\Aut(\cC)=\Aut(\cH(g)).$
\end{defn}

We have the following:

\begin{lem}\label{l:3tournauto}
Let $\cC$ be a switching class of $3$-tournaments on $V$ and let $\{g_1,g_2,\ldots, g_k\}$ be a complete set of isomorphism class representatives of elements of $\cC$. There is a one to one correspondence $\Aut(\cC)/\Aut(g_1) \rightarrow \{g \in \cC \mid g \cong g_1\}$ and $$\sum_{i=1}^k \frac{1}{|\Aut(g_i)|} = \frac{2^{{|V|-1}\choose{2}}}{|\Aut(\cC)|}.$$
\end{lem}

\begin{proof}
This follows from very similar arguments to those given in the proofs of Lemmas \ref{l:cosets} and \ref{l:cosets2} for a switching class of 2-tournaments. 
\end{proof}

\begin{exmp}\label{e:3tournex}
Up to isomorphism, there are six $3$-tournaments $g$ on $\{1,2,3,4,5\}$ for which $\cH(g)=\{\{1,2,3,4\},\{1,2,3,5\}\}$ with automorphism groups of orders $1,1,1,1,1,3$. We have $\Aut(\cH(g)) \cong \Sym(\{1,2,3\}) \times \Sym(\{4,5\})$ and $\displaystyle 5+\frac{1}{3} = \frac{2^6}{12}$, consistent with Lemma \ref{l:3tournauto}.
\end{exmp}


\begin{defn}
Let $V$ be a set. A permutation $\sigma \in \Sym(V)$ is \textit{level} if the powers of $2$ dividing its cycle lengths are all equal. If $G \le \Sym(V)$, write $L_G$ for the set of all level permutations in $G$.
\end{defn}

\begin{defn}
Let $V$ be a set and $\sigma \in \Sym(V)$.
\begin{enumerate}
\item The number of cycles in $\sigma$ is denoted $\orb(\sigma)$;
\item The number of orbits of $\langle \sigma \rangle$ on unordered pairs  is denoted $\orb_2(\sigma)$;
\item $\delta: \Sym(V) \rightarrow \{0,1\}$ is the function $$\delta(\sigma)=\begin{cases} 0 & \mbox{ if all cycles of $\sigma$ have even length } \\ 1 & \mbox{ otherwise.} \end{cases}$$
\end{enumerate}
\end{defn}

The following is a generalisation of \cite[Theorem 7.2]{BC2000}:

\begin{thm}\label{3tourniso}
The number of isomorphism classes of elements of a switching class $\cC$ of $3$-tournaments on $V$  is: $$\frac{1}{|\Aut(\cC)|} \sum_{\sigma \in L_{\Aut(\cC)}} 2^{\orb_2(\sigma)-\orb(\sigma)+\delta(\sigma)}.$$
\end{thm}

\begin{proof}
The number of $3$-tournaments fixed by a permutation $\sigma \in \Aut(\cC)$ is the number of two-graphs on $V$ that it fixes. If $\sigma$ is not level then there is $r \in \NN$ for which $\sigma^r$ contains a transposition $(x,y)$ and a fixed point $z$. Hence for any $3$-tournament $g$ $g\sigma \neq g$ since $g\sigma^r(x,y,z) =g(y,x,z) =-g(x,y,z) \neq g(x,y,z)$. By \cite{MS75}, $\sigma$ fixes $2^{\orb_2(\sigma)-\orb(\sigma)+\delta(\sigma)}$ two-graphs so the result follows from the Orbit Counting Lemma.
\end{proof}

\begin{exmp}
If $|\Aut(\cC)|=1$ then Theorem \ref{3tourniso} implies that the number of isomorphism classes of elements of $\cC$ is $$2^{\orb_2(1)-\orb(1)+\delta(1)}=2^{{|V| \choose 2}-|V|+1}=2^{{|V|-1} \choose 2}$$
which is the number of two-graphs on $V$, as predicted by Lemma \ref{l:3tournauto}.
\end{exmp}

\begin{exmp}
If $\cC$ is the switching class of the $3$-tournament in Example \ref{e:3tournex}) then $L_{\Aut(\cC)}=\langle (1,2,3) \rangle$ and using Theorem \ref{3tourniso} we readily calculate that there are $$\frac{1}{12}(2^{10-5+1}+2 \cdot 2^{4-3+1})=\frac{72}{12}=6$$ isomorphism classes of elements of $\cC$, as expected.
\end{exmp}



\end{document}